\documentclass[12pt]{amsart}

\usepackage{amsmath}
\usepackage{cite}

\usepackage{appendix}
\usepackage{kotex}
\usepackage{lineno,hyperref}
\modulolinenumbers[5]
\usepackage{epstopdf}
\usepackage{pifont}
\usepackage{latexsym}
\usepackage{graphics}
\usepackage{graphicx}
\usepackage{float}
\usepackage[dvips]{xy}
\xyoption{all}
\usepackage{ifpdf}
\usepackage{multirow}
\usepackage{amssymb, amsthm}
\usepackage{hhline}
\usepackage{centernot}
\usepackage{verbatim, comment, color}
\usepackage{enumerate}
\usepackage{setspace}
\usepackage{lipsum}
\usepackage{subcaption}

\usepackage{mathtools}

\newtheorem{theorem}{Theorem} 
\newtheorem{proposition}[theorem]{Proposition}
\newtheorem{lemma}[theorem]{Lemma}
\newtheorem{remark}[theorem]{Remark}
\newtheorem{corollary}[theorem]{Corollary}

\newtheorem{example}[theorem]{Example}

\newcommand{\ba}{\begin{align}}
\newcommand{\ea}{\end{align}}  

\newcommand{\be}{\begin{equation}}

\newcommand{\ee}{\end{equation}}
\newcommand{\bea}{\begin{eqnarray}}
\newcommand{\eea}{\end{eqnarray}}
\newcommand{\barr}{\begin{array}}
\newcommand{\earr}{\end{array}}
\newcommand{\bn}{\begin{enumerate}}
\newcommand{\en}{\end{enumerate}}
\newcommand{\bi}{\begin{itemize}}
\newcommand{\ei}{\end{itemize}}
\newcommand{\bbbm}{\begin{pmatrix}}
\newcommand{\eeem}{\end{pmatrix}}

\newcommand{\bbN}{{\bf N}}

\newcommand{\bbS}{{\bf S}}

\newcommand{\cE}{{\cal E}}

\newcommand{\cH}{{\cal H}}

\newcommand{\cP}{{\cal P}}

\newcommand{\R}{{\mathbf R}}

\newcommand{\al}{\alpha}
\newcommand{\bt}{\beta}

\newcommand{\de}{\delta}

\newcommand{\ep}{\epsilon}

\newcommand{\si}{\sigma}

\newcommand{\tta}{\theta}
\newcommand{\ignore}[1]{}{}

\newcommand{\nn}{\nonumber}

\newcommand{\p}{{\partial}}

\newcommand{\q}{\quad}

 \newcommand{\Id}{\mathop{\rm Id}}

\newcommand{\Prob}{{\mathcal P}}
\newcommand{{\QED}}{{\hfill QED} \smallskip}

\newcommand{\Rn}{{\R^n}}
\newcommand{\spt}{\mathop{\rm spt}}

\newcommand{\Wa}{W_\alpha}
\newcommand{\Wb}{{W_\beta}}
\newcommand{\Wab}{W_{\alpha,\beta}}

\renewcommand{\subset}{\subseteq}

\newcommand{\cal}{\mathcal}

 \DeclareMathOperator*{\argmin}{argmin}
  
\DeclareMathOperator*{\Var}{{Var}}

  \definecolor{darkspringgreen}{rgb}{0.09, 0.45, 0.27} 
 \definecolor{darkgray}{rgb}{0.66, 0.66, 0.66}

\numberwithin{equation}{section}
\numberwithin{theorem}{section}
    
\begin{document}
\title[Classifying minimizers of attractive-repulsive interactions]
{Classifying minimum energy states for interacting particles: 
spherical Shells
 }
\thanks{
\em \copyright 2022 by the authors. The authors thank Dejan Slepcev {and an anonymous referee} for fruitful suggestions, {and Rupert Frank for pointing out prior use of Proposition \ref{P:Lopes Formula}
 in his work \cite{CarrilloDelgadinoDolbeaultFrankHoffman19} with Carrillo et al.}  CD acknowledges the support of a 
Natural Sciences and Engineering Research Council of Canada Undergraduate Research Grant.
 TL is grateful for the support of ShanghaiTech University, and in addition, to the University of Toronto and its Fields Institute for the Mathematical
Sciences, where parts of this work were performed.  RM  acknowledges partial support of his research by the Canada Research Chairs Program and
Natural Sciences and Engineering Research Council of Canada Grant {2020-04162.}}

\date{\today}

\author{Cameron Davies, Tongseok Lim and Robert J. McCann}
\address{Cameron Davies: Department of Mathematics \newline University of Toronto, Toronto ON Canada}
\email{cameron.davies@mail.utoronto.ca}
\address{Tongseok Lim: Krannert School of Management \newline  Purdue University, West Lafayette, Indiana 47907, USA}
\email{lim336@purdue.edu}
\address{Robert J. McCann: Department of Mathematics \newline University of Toronto, Toronto ON Canada}
\email{mccann@math.toronto.edu}

\begin{abstract} 
Particles interacting through long-range attraction and short-range repulsion given by power-laws have been widely used to model physical and biological systems, and to predict or explain many of the patterns they display. Apart from rare values of the attractive and repulsive exponents $(\al,\bt)$, the energy minimizing configurations of particles are not explicitly known,  although simulations and local stability considerations have led to conjectures with strong evidence over a much wider region of parameters. For {dimension $n\ge 2$}, and for a segment $\bt=2<\al<4$ on the mildly repulsive frontier we employ strict convexity to conclude that the energy is uniquely minimized ($d_\infty$-locally, up to translation) by a spherical shell. If $n=1$ and $\beta=2<\alpha-1$, we prove that the spherical shell is (i) the unique global energy minimizer, and (ii) the unique $d_\infty$-local energy minimizer in the class of even, compactly-supported probability measures.
 In a companion work, we show that in the mildly repulsive range $\al>\bt\ge2$, a unimodal threshold $2<\al_{\Delta^n}(\bt) \le \max\{\bt,4\}$ exists such that equidistribution of particles over a unit diameter regular $n$-simplex minimizes the energy if and only if $\al  \ge \al_{\Delta^n}(\bt)$ (and minimizes uniquely up to rigid motions if strict inequality holds). {For $n\ge 2,$} the point $(\al,\bt)=(2,4)$ separates these regimes{. At this point} we show the minimizers all lie on a sphere and are precisely characterized by sharing all first and second moments with the spherical shell. Although the minimizers need not be asymptotically stable, our approach establishes $d_\al$-Lyapunov nonlinear stability of the associated ($d_2$-gradient) aggregation dynamics near the minimizer in both of these adjacent regimes --- without reference to linearization. The $L^\al$-Kantorovich-Rubinstein{-Wasserstein} distance $d_\al$ which quantifies stability is chosen to match the attraction exponent.
\end{abstract}

\maketitle
\noindent\emph{Keywords:  aggregation equation, spherical shell, attractive-repulsive power-law interaction, convex, unique energy minimizer,  Lyapunov stability, 
asymptotic stability, Kantorovich-Rubinstein-Wasserstein distance, infinite-dimensional quadratic programming,
$d_\infty$-local minimum
}

\noindent\emph{MSC2020 Classification: Primary 49Q10. Secondary 31B10, 35Q70, 37L30, 70F45, 90C20} 

\section{Introduction}
 The self-interaction energy of a collection of particles with mass distribution $d\mu(x) \ge 0$ on $\R^n$ is given by 
 \ba \label{energy}
 \mathcal{E}_W(\mu) =  \frac12 \iint_{\R^n \times \R^n} W(x-y) d\mu(x)d\mu(y),
\end{align}
assuming the particles interact with each other through a pair potential $W(x)$.   Normalizing the distribution to have unit mass ensures that $\mu$ belongs to the space $\Prob(\R^n)$ of Borel
probability measures on $\R^n$. 

Our goal is to identify global energy minimizers
of $\cE_W(\mu)$ on $\Prob(\Rn)$,  for {\em power-law} potentials $W=\Wab$ where
\begin{align} \label{potential}
\Wa(x) &:= |x|^\al/\al  \q {\rm and}
\\ \Wab(x) &:= \Wa(x) - \Wb(x), 
\q  {-n<\bt<\al<\infty.} 
\label{potential2}
\end{align}
When $\bt \ge 2$ the potential is called {\em mildly repulsive} \cite{CarrilloFigalliPatacchini17}. In this paper, we focus on the mild repulsion threshold $\bt=2$  {called the centrifugal line in \cite{LimMcCann21}, since, at least on $\R^2$, the potential $-W_2$ induces the outward force 
which particles rotating uniformly around their common center of mass seem to experience in a corotating reference frame;
see e.g.~\cite{McCann06}.
When $\bt=2$ the energy also acts as a Lyapunov function of the rescaled dynamics of the purely attractive Patlak-Keller-Segel 
model \cite{Patlak53} \cite{KellerSegel70} in self-similar variables around the time of blow-up \cite{SunUminskyBertozzi12}.}
{If $\al \in (2,4)$, we will show that the minimizer is uniquely given (up to translations) by a spherical shell, i.e. the uniform probability measure on a spherical hypersurface of the appropriate radius.   
For $\al>4$ and $\bt \ge 2$,  we build on these results to show in a companion
paper that the minimizer is uniquely given (apart from rotations and translations) by equidistributing its mass over the vertices of a regular $n$-simplex.  
Together, these results resolve some questions left open by Sun, Uminsky and Bertozzi
by showing that the linear stability of selfsimilar blow-up which they found 
for the aggregation dynamics 
in these two regimes can be improved to a nonlinear Lyapunov stability result.
On the other hand, at the threshold exponent separating these two regimes, 
we will show that although all centered convex combinations of the configurations mentioned above remain mimimizers, 
there are many additional minimizers as well: indeed for $(\al,\bt)=(4,2)$ the centered minimizers 
consists of those measures supported on the minimizing spherical shell which share its moments up to order 2. When $n \ge 2$, this case is distinguished from $\al \ne 4$ by the fact that the Lyapunov stable set formed by global energy minimizers becomes infinite-dimensional.

To understand the literature surrounding these questions,  we recall that heuristically, the aggregation equation
\be\label{aggregation}
\frac{\p \mu}{\p t} = \nabla \cdot (\mu \nabla (W*\mu))
\ee
arises as the $d_2$-gradient descent of the energy \eqref{energy} with respect to the Kantorovich-Rubinstein-Wasserstein metric
\be\label{KRW metric}
d_p(\mu,\nu) 
:=\inf_{X \sim \mu, Y \sim \nu} \| X - Y\|_{L^p},
\end{equation}
defined for $p\in [1,\infty]$ on probability measures $\mu,\nu \in \Prob(\R^n)$.  
Here $X \sim \mu$ denotes a random vector in $\Rn$ with law $\mu$, and the infimum is over all pairs of random vectors with {fixed laws $\mu$ and $\nu$ (respectively)}.
In the mildly repulsive regime,  $W_{\al,\bt}$ is semiconvex and this heuristic inspired by  \cite{JordanKinderlehrerOtto98} can
 be made rigorous \cite{AmbrosioGigliSavare05} \cite{CarrilloMcCannVillani06} \cite{Villani03}:
 the evolution \eqref{aggregation} is well-posed in the space of probability measures having finite second moments.  
 Under the flow which results,  the energy \eqref{energy} is non-increasing;  we shall show below that the family of global energy
 minimizers forms a $d_\al$-Lyapunov stable family of fixed points of the evolution, where the power $p=\al$ quantifies this stability in terms of the attraction exponent.  
Steady-state examples of discrete particle rings  \cite{BertozziKolokolnikovSunUminskyVonBrecht15} \cite{KolokolnikovSunUminskyBertozzi11} approximating the minimizer show
that the spherical shell will not be asymptotically stable (i.e.~does not form an attractor), 
in spite of the fact  \cite{BalagueCarrilloLaurentRaoul13N} that $d_p$-asymptotic stability 
holds locally in the more restricted class of {\em spherically symmetric initial data for some $p \ge 1$;} 
c.f.~Example~\ref{E:ring}.  
 For $\al>\bt>2$ there are {uncountably} many $d_\infty$-local minima \cite{LimMcCann21} \cite{Simione14} 
--- which we also expect to be asymptotically stable fixed points of the evolution.
Dynamics analogous to \eqref{aggregation} have been proposed as models for the kinetic flocking and swarming behaviour of biological organisms \cite{Breder54}
\cite{MogilnerEdelstein-Keshet99} \cite{TopazBertozziLewis06}, condensation of granular media \cite{BenedettoCagliotiPulvirenti97} \cite{Toscani00} 
\cite{CarrilloMcCannVillani03}, self-assembly of nanomaterials \cite{HolmPutkaradze06},
and even strategies in game theory \cite{BlanchetCarlier14}.  
For this reason, they have often been simulated and a wide variety of patterns have been observed to emerge, depending on $(\al,\bt)$ 
and initial conditions 
\cite{AlbiBalagueCarrilloVonBrecht14} \cite{BertozziKolokolnikovSunUminskyVonBrecht15} \cite{CraigBertozzi16} \cite{KolokolnikovSunUminskyBertozzi11} \cite{vonBrechtUminskyKolokolnikovBertozzi12}.

Despite much attention,  there are relatively few cases in which the global minimum of 
\eqref{energy} over $\cP(\Rn)$ is known explicitly \cite{LimMcCann21},
and many of these either involve additional effects such as diffusion  \cite{CarrilloHittmeirVolzoneYao19} \cite{DelgadinoYanYao20}
or density bounds \cite{BurchardChoksiTopaloglu18} \cite{FrankLieb18} \cite{FrankLieb19+},
or fall outside the mildly repulsive regime \cite{BurchardChoksiHess-Childs20} \cite{CarrilloHuang17} \cite{ChoksiFetecauTopaloglu15}
 \cite{FetecauHuang13} \cite{FetecauHuangKolokolnikov11}.
Several groups of authors have explored how properties of the mimima, such as dimension of its support \cite{BalagueCarrilloLaurentRaoul13} \cite{CarrilloFigalliPatacchini17},
vary with the exponents $(\al,\bt)$.  Others have investigated nonlinear stability of the steady states locally.
Following work in one-dimension by Fellner and Raoul \cite{FellnerRaoul11},
for $\bt >-n$, Balagu\'e, Carrillo, Laurent and Raoul \cite{BalagueCarrilloLaurentRaoul13N} have shown the sign of $\bt-\bt^*$  to determine nonlinear stability ($\bt>\bt^*$) or 
instablity ($\bt<\bt^*$) of
the spherical shell of radius
\be\label{Rab}
R= R_{\al,\bt} = \frac12 \left[\frac{\Gamma(\frac{\bt+n-1}{2}) \Gamma(\frac\al2 + n-1)}{\Gamma({\frac\bt2 + n-1}) \Gamma({\frac{\al+n-1}2})} \right]^{\frac1{\al-\bt}},
\ee
among $d_\infty$-small spherically symmetric perturbations, where 
\[
\bt^* := \frac{(3-n)\al - 10 + 7 n - n^2}{\al +n -3},
\]
{and $\Gamma(\cdot)$ is Euler's Gamma function, \eqref{Euler product}.}
Although  it lies outside the mildly repulsive regime,  for $\bt<\bt^*$, $R_{\al,\bt}$ remains the unique radius \eqref{Rab} at which a spherical shell is a steady state.
 Families of convex combinations of spherical shells form an invariant family under the flow \eqref{aggregation},  on which the dynamics reduces to a system of ordinary differential
equations analyzed by Balagu\'e Guardia, Barbaro, Carrillo and Volkin \cite{BalagueBarbaroCarrilloVolkin20}.
Less has been shown about the dynamics of radial measures with non-singular densities however.
For perturbations which destroy spherical symmetry,  the absence of a spectral gap makes local stability of steady states a much subtler issue.  The asymptotic stability of steady state spherical shells in certain spaces might be bootstrapped from linear stability using the framework of  von Brecht and McCalla \cite{vonBrechtMcCalla14}, while for more general steady states including some supported on the discrete two dimensional rings of Example \ref{E:ring}
$d_\infty$-Lyapunov and $d_2$-asymptotic stability have also been addressed  
by Simione \cite{Simione14}.

\section{Results}

Let us preface our results with a proposition reviewing the existence \cite{ChoksiFetecauTopaloglu15} and some relevant properties of energy minimizers.
\begin{proposition}[Minimizers]
\label{exist} 
For $\al > \bt > 0$, minimizers of $\mathcal{E}_{W_{\al,\bt}}$ on $\cP(\Rn)$ exist and the diameter of their supports is uniformly bounded by $e^{1/\bt}$.
Moreover, each such minimizer $\mu$ satisfies
\be\label{Euler-Lagrange}
\mu[\argmin_{\Rn} (W_{\al,\bt} * \mu)]=1
\ee
\end{proposition}
\begin{proof}
For $\al > \bt > 0$, \cite[Lemma 1]{KKLS21} shows the diameter of support of all ($d_2$-local) minimizers for $ \mathcal{E}_{W_{\al,\bt}}$ is bounded by the positive zero, say $z_{\al,\bt}$, of the function $w_{\al,\bt}(r)=r^\al/\al - r^\bt/\bt$. It is easily seen that $z_{\al,\bt}$ increases as $\al \searrow \bt$ to the limit $e^{1/\bt}$. Now to show the existence of minimizers, consider the set $\cP\left(\overline{ B_{e^{1/\bt}}}\right)$ of probability measures concentrated in the centered closed ball $\overline{B_{e^{1/\bt}}}$ of radius $e^{1/\bt}$. As $\cP\left(\overline{ B_{e^{1/\bt}}}\right)$ is weakly compact and $\mu \in \cP\left(\overline{B_{e^{1/\bt}}}\right) \mapsto  \mathcal{E}_{W_{\al,\bt}}(\mu)$ is weakly continuous, this energy must attain a minimizer. By the a priori diameter estimate mentioned above, this minimizer on
$\cP\left(\overline{B_{e^{1/\bt}}}\right)$ also minimizes $ \mathcal{E}_{W_{\al,\bt}}$ among probability measures on $\R^n$. The Euler-Lagrange equation \eqref{Euler-Lagrange} for minimizers 
is established e.g.~in \cite{BalagueCarrilloLaurentRaoul13}.
\end{proof}
Let a {\em (centered) spherical shell}  denote the uniform probability measure $ \sigma_R$ 
on a (centered) sphere of radius $R \ge 0$. 
Also, let 
\be\label{centered}
\cP_0(\R^n)=\{ \mu \in \cP(\R^n) \ | \ \int |x| d\mu(x) < \infty, \int x d\mu(x) = 0\}
\ee
denote the set of probability measures with center of mass at the origin,
and 
$
\cP^{ss}_c(\R^n) \subset \cP_0(\R^n) 
$ 
the subset consisting of spherically symmetric measures (i.e. those invariant under the action of the orthogonal group $O(n)$) having compact support.
The first contribution of the present manuscript is to establish the following result for potentials $W_{\al,\bt} *\mu$ of spherically symmetric measures in an interval $(\al,\bt) \in (2,4) \times \{2\}$ along the mildly repulsive frontier:

\begin{theorem}[Radial potentials have a single inflection point]\label{T:main}
Fix 
\begin{equation}\label{range}
\begin{cases}
\bt = 2<\al<4 & {\rm if}\ n \ge 2, 
\\ {\bt = 2< \al -1} & {\rm if}\ n=1,   
\end{cases}
\end{equation}
and let $e_1 \in \R^n$ denote a unit vector. 
If $\mu \in \cP^{ss}_c(\R^n)$ then $f(r) = f_\mu(r) := (W_{\al,\bt}*\mu)(r e_1)$ satisfies
$f'''(r)>0$ for all $r>0$; moreover
$f \in C^3_{loc}(\R \setminus \{0\}) \cap C_{loc}^{\lfloor \alpha\rfloor,\al-\lfloor \al \rfloor}(\R)$,
where $\lfloor \al \rfloor$ denotes the largest integer $k\le \al$. In particular,
there exists $0 \le R <\infty$ such that $f(r)$  is strictly concave on the interval $|r| \le R$, and strictly convex on $r>R$. 
 \end{theorem}

The core of its proof is to establish existence and positivity of the third derivative $f'''(r)>0$ for all $r>0$.
If $R=0$ the asserted strict concavity is vacuous, and $f$ is strictly convex globally.

Since the potential $W_{\al,\bt}*\mu$ appears in the Euler-Lagrange equation \eqref{Euler-Lagrange},
it is no surprise that this theorem implies the following corollary 
concerning stationary solutions of the evolution \eqref{aggregation}.
While statement (a) assumes spherical symmetry,  all global minimizers (c) inherit this symmetry from the strict convexity
established by Lopes \cite{L19} for the energies \eqref{range} on {\em centered} probability densities, 
which extends to singular measures {and $d_\infty$-local minimizers as in 
\cite{CarrilloDelgadinoDolbeaultFrankHoffman19} \cite{CarrilloShu21+}} and Appendix~\ref{S:Strict convexity} below.

\begin{corollary}[Classifying energy minimizers and stationary solutions on part of the centrifugal line]
\label{C:abc}
Assume {$\bt=2$ and $(\al,n)$} satisfy 
\eqref{range}. 
\\ (a) If $\nabla (W_{\al,\bt} * \mu)=0$ holds $\mu$-a.e. for some $\mu \in \cP^{ss}_c(\R^n)$, then
$\mu = (1-s)\sigma_{r} + s \delta_0$,  where $s\in [0,1]$ and $\sigma_{r}$ is the uniform probability on the sphere of radius $r>0$.
\\ (b) If $n \ge 2$ and $\mu$ minimizes the energy $\mathcal{E}_{W_{\al,\bt}}$ on some open $d_\infty$-ball 
in $\cP_c(\Rn)$, 
 then $\mu$ is a translate of the spherical shell $\sigma_{r_*}$ of radius 
$r_* =\big{(}\frac{2}{ f'_{\si_1}(1)+2} \big{)}^{\frac1{\al-2}}$, where $f_{\sigma_1}(r)$ is defined in Theorem \ref{T:main}.  
\\ (c) For $n \ge 1$, the energy $\mathcal{E}_{W_{\al,\bt}}$ is minimized uniquely on $\cP(\Rn)$  by the spherical shell $\si_{r_*}$ and its translates.
\\ (d) Conclusion (c) extends also to the limit case $(\al,\bt,n)=(3,2,1)$;
\\ (e) If $n=1$ and $\mu$ minimizes the energy $\mathcal{E}_{W_{\al,\bt}}$ on some open $d_\infty$-ball 
in the space $\cP^{ss}_c(\Rn)$ of even measures, 
 then $\mu=\sigma_{r_*}$.
\end{corollary}

A spherical shell fails to minimize $ \mathcal{E}_{W_{\al,\bt}}$ if $\bt=2 < \al <3$ and $n=1$ however; see \cite{Frank22} 
for the minimizer and Remark \ref{1D references}.

Note that Corollary (c) 
can also be
proved by using Theorem \ref{T:main} to verify that some spherical shell satisfies
the Euler-Lagrange/Kuhn-Tucker conditions \eqref{Euler-Lagrange}.  By the energetic convexity, this is not only necessary but also sufficient to identify a global minimum.  The problem displays an unexpected subtlety however: normally, a convex gradient-flow cannot display non-minimizing stationary solutions,  such as those described in Example \ref{E:ring}.
Evidently these discrete ring solutions $\mu$, although stationary, 
can only satisfy a localized version of the Euler-Lagrange / Kuhn-Tucker conditions \eqref{Euler-Lagrange}, 
in which   $W*\mu$ attains a local but not a global minimum on $\spt \mu$. 

In the complementary range
\begin{equation}\label{CSrange}
(\alpha,\beta) \in [2,4] \times (-n,2)
\end{equation}
a different classification problem was solved  by
Carrillo and Shu \cite{CarrilloShu21+}
who established that
all compactly supported $d_\infty$-local minimizers are spherically symmetric. 
(In a more restricted range of exponents,  they are able to give a precise description of the global minimizers:
on the {\em centripetal} line $4-n<\bt < 2=\al$ their minimizers  turn out to be spherical shells.) Stimulated by an anonymous referee's query,  we improve Corollaries \ref{C:abc}(b)(e) and \ref{radialmin}, which originally addressed only {minimizers relative to a symmetry constraint,  to include $d_\infty$-local minima as well. We do this by extending Carrillo and Shu's classification to the relative interior $(2,4) \times \{2\}$ of the corresponding interval on the centrifugal line when $n \ge 2$ (so nontrivial rotations exist).
However, the next theorem shows no such extension can hold at the right endpoint  $(\al,\bt)=(2,4)$;   it also fails at the left endpoint where $W_{2,2}=0$. 

The results of our companion paper \cite{DaviesLimMcCann21+b}, which also extend to the interior $\al>\bt$ of the mildly repulsive regime $\bt \ge 2$,  allow us to complete this characterization of minimizers on the frontier $\bt=2$ as follows,  at least for  $n \ge 2$.  
A set $K \subset \R^n$ is called a {\rm unit $ n$-simplex} if it is the convex hull of $ n+1$ points $\{x_0, x_1,...,x_{n}\}$ in $\R^n$ satisfying $|x_i-x_j|=1$ for all $0 \le i < j \le n$. 
 The points $\{x_0, x_1,...,x_{n}\}$ are called {\em vertices} of the simplex. We define 
 \begin{align}\label{simplex}
\cP_{\Delta^n} & :=\{\nu \in \cP(\R^n) \ | \ \nu \text{ is uniformly distributed}\\ &\qquad \qquad \text{over the vertices of a unit $n$-simplex.} \} \nn
\end{align}

\begin{theorem}[Simplices uniquely minimize energy over much of the mildly repulsive regime \cite{DaviesLimMcCann21+b}]
\label{b=2}
Let $n \ge 2=\bt$.\\
(i) If $2<\al <4$, then the {unique} minimizer of $\mathcal{E}_{W_{\al,2}}$ on \eqref{centered}  is given by a spherical shell of {positive} radius. 
\\
(ii) If $\al=4$, then $\mu \in \cP_0(\R^n)$ minimizes  $\mathcal{E}_{W_{\al,2}}$ 
if and only if $\mu$ is concentrated on the centered sphere of radius $ \sqrt{\frac{n}{2n+2}}$ with 
\be\label{tensor}
\int x \otimes x \, d\mu(x) = \bigg( \int x_ix_j d\mu(x)\bigg)_{1 \le i,j \le n}= \frac{1}{2n+2} \Id,
\ee
where $\Id$ denotes the $n \times n$ identity matrix.\\
(iii) If $\al >4$, the set of minimizers of  $\mathcal{E}_{W_{\al,2}}$ 
is precisely $\cP_{\Delta^n}$.
\end{theorem}

In the next section we prove Theorem \ref{T:main} and its corollaries.  
A subsequent chapter discusses nonlinear stability implications for the evolution \eqref{aggregation}
near minimizers such as the spherical shell.  This is followed by an appendix
extending a
 strict convexity result for the energy shown by Lopes \cite{L19} for densities to measures,
 also obtained independently in \cite{CarrilloDelgadinoDolbeaultFrankHoffman19}.

\section{Spherical shells minimize for $\bt = 2<\al<4$} 
Denote by $\si_R \in \cP_0(\R^n)$ the uniform measure on a sphere of $R \ge 0$ --- called a {\em spherical shell} --- and let $\si:=\si_1$ denote the unit spherical shell. 
For $y \in \R^n$, denote $y_i = e_i \cdot y$ for $i=1,...,n$, where $\{e_i\}_{1 \le i \le n}$ is the standard basis of $\R^n$.
{In this section we shall establish Theorem~\ref{T:main}} whose proof uses the following
elementary lemma to establish positivity of the radial third derivative of the potential $W_{\al,2}*\mu$ induced by any spherically symmetric measure $\mu \in \cP^{ss}_c(\Rn)$.
Positivity of this third derivative
shows that as $|x|$ increases, the radial profile of $W_{\al,2}*\mu$ can transition from concave to convex but not from convex to concave.  
Thus it is minimized by a unique positive radius.  The Euler-Lagrange equation \eqref{Euler-Lagrange} satisfied by a minimizer then forces $\mu$ to be a spherical shell
(or a convex combination of such a shell with a Dirac measure at the origin, which requires an additional argument to rule out).
The next lemma extends the positivity of \eqref{gal} shown for $\al<4$ by \cite[Lemma 4.4]{Dong11} to larger values of $\al$, c.f. \cite{BalagueCarrilloLaurentRaoul13N}  \cite{BalagueBarbaroCarrilloVolkin20} \cite{CarrilloShu21+},
as well as establishing \eqref{Gal}. To prove Theorem \ref{T:main}, we shall eventually 
show the third derivative \eqref{f3} of $f(r) = (W_{\al,2}*\mu)(re_1)$ to be a positive combination of such integrals when $\mu=\sigma$ is a spherical shell.

\begin{lemma}[Positivity of certain spherical integrals]
\label{positive}
Let $n \ge 2$. Then for all $\al > 2$, the following integrals are absolutely convergent for any $r \in \R$, and define  continuous odd functions $g_\al$ and $G_\al$ on $\R$. Moreover for all $r > 0$ they are positive:
\begin{align}
\label{gal} 
g_\al(r) & :=\int (r-y_1) | re_1 -y |^{\al-4} d\si(y) > 0 \quad {\rm and}
\\ \label{Gal}
G_\al(r) &:=\int  | re_1 -y |^{\al-6}(r-y_1)(1-y_1^2)d\si(y) > 0.
\end{align}
\end{lemma}

\begin{remark}[Non-negativity of the limiting integrals]
If instead $\al=2$, the same proof shows $g_\al(r) \ge 0$ and $G_\al(r) \ge 0$ unless $r=1=n-1$ (in which case the integrals 
\eqref{gal} and \eqref{Gal} no longer converge absolutely). Equality holds in either (hence both) inequalities if and only if $n=\al=2$ and $|r|<1$.  This endpoint case is reminiscent of Newton's shell theorem.
\end{remark}

\begin{proof}
When $r=1$, both integrands 
have exponent $\al -3 > -(n-1)$ {on the unit sphere in $\R^n$, hence the integrals are absolutely convergent. It then follows from the dominated convergence theorem that}  $g_\al$ and $G_\al$ are well-defined and continuous in $\R$, and odd. Moreover it is clear that both integrals are positive for all $r \ge 1$. So let us fix $r \in (0,1)$. For \eqref{gal} we claim that it is sufficient to show $g_2(r) \ge 0$. To see this, let us rewrite $g_\al$ by scaling. Let $s=s(r) >1$ be defined by $s^2 - s^2r^2=1$. Then
\begin{align} \label{scaling}
g_\al&=\int  | re_1 -y |^{\al-4}(r-y_1) d\si(y)   \\
 &=   s^{3-\al}\int  | sre_1 -y |^{\al-4}(sr-y_1)d\si_s(y).  \nn
\end{align}
Hence
\begin{align*}
\frac{\partial g_\al}{\partial \al} = &-s^{3-\al} (\log s) \int  | sre_1 -y |^{\al-4}(sr-y_1)d\si_s(y) \\
&+s^{3-\al} \int  | sre_1 -y |^{\al-4} (\log  | sre_1 -y | ) (sr-y_1) d\si_s(y).
\end{align*}
Observe the second integral is positive because our choice of $s=s(r)$ makes the sign of $sr-y_1$ and $\log  | sre_1 -y |$ coincide. Since the first integral vanishes precisely when $g_\al=0$,  we see that
\begin{align}\label{pode}
{g_\al \le 0} \ \text{ implies } \ \frac{\partial g_\al}{\partial \al} >0
\end{align}
which in turn shows
\begin{align}\label{boot}
g_{\al_0} \ge 0 \ \text{ implies }  \ g_\al >0 \text{ for all } \al > \al_0
\end{align}
to prove the claim. It remains to show $g_2(r) \ge 0$.
Let $\cH^{n-1} $ denote the area measure on the unit sphere in $\Rn$, so that $\omega_n d\si = d\cH^{n-1}|_{\bbS^{n-1}}$ where
\[
\omega_n := \cH^{n-1}[\bbS^{n-1}] (= \frac{2 \pi^{n/2}}{\Gamma(\frac n2)}).
\]
Expressing the Euclidean volume element $d^n x = r^{n-1} dr \prod_{i=1}^{n-1} (\sin \theta_i)^{i-1} d\theta_i$ in polar coordinates yields
\begin{align}
g_\al(r)&=\frac{1}{\omega_n} \int_{{\bf S}^{n-1}} | re_1 -y |^{\al-4}(r-y_1) d\cH^{n-1}(y) \nn \\
\label{spherical polars}&{=}\frac{\omega_{n-1}}{\omega_n} \int_0^\pi \frac{r-\cos \theta}{(r^2 - 2 r \cos\theta + 1)^{(4-\al)/2}} (\sin\theta)^{n-2} d\theta. 
\end{align}
In the special case $n=2=\al$ we claim the integral \eqref{spherical polars} vanishes.
Indeed,  exchanging the order of the integrals yields
\[
\frac{\omega_2}{\omega_1} \int_0^r g_2(s)ds 
= \int_0^\pi \log| r^2 - 2 r \cos\theta + 1| d\theta.
\]
The last integral  
is known to vanish, e.g. \cite[p 104]{SteinShakarchi03}.
Thus $g_2(r)=0$ for $|r|<1$ when $n=2$.

The fact that $g_2(r)>0$ when $n>2$ can be argued as follows.
Taking $\al=2$, the symmetry $\cos(\pi -\theta)=-\cos\theta$ reduces \eqref{spherical polars} to 
\begin{align}\label{spherical polar symmetry}
\frac{\omega_n}{\omega_{n-1}} g_2(r) 
&=2r  \int_0^{\pi/2} \frac{r^2+1 - 2 \cos^2 \theta}{(r^2+1)^2 - 4 r^2 \cos^2\theta} (\sin\theta)^{n-2} d\theta.
\end{align}
The
integrand in \eqref{spherical polar symmetry} changes sign only once,
at $\theta_r := \arccos\sqrt{\frac{r^2+1}{2}}$.  
Since the weight 
$\sin^{n-2}\theta$ 
(uniform when $n=2$) becomes an increasing function of $\theta \in [0,\frac \pi 2]$ for $n>2$,
it suppresses the contributions from the negative region $(0,\theta_r)$ and enhances the contributions
from the positive region $(\theta_r,\frac\pi 2)$ relative to the unweighted case.
 Thus positivity $g_2(r)>0$ of the weighted integral 
 for $n>2$  follows from the fact that the unweighted integral vanishes when $n=2$.

 Having established \eqref{gal} we turn now to \eqref{Gal}.
As before, it is clear that $G_\al(r) >0$ for all $r \ge 1$. So let us fix $r \in (0,1)$. By a similar scaling argument to 
\eqref{scaling}, \eqref{pode}, \eqref{boot}, we deduce
\begin{align}
G_{\al_0}(r) \ge 0 \ \text{ implies } \ G_\al(r) >0 \ \text{ for all } \ \al > \al_0.
\end{align}
Thus we only need to show $G_2(r) \ge 0$ (with equality when $n=2$).

Re-expressing \eqref{Gal} in spherical polar coordinates analogously to \eqref{spherical polars} 
and using  the symmetry $\cos(\pi -\theta)=-\cos\theta$ yields that, for $\al=2$,
\begin{align}\label{spherical polars 2}
{\frac{\omega_n}{\omega_{n-1}}
} 
G_2(r) 
&= \int_0^\pi \frac{r-\cos \theta} {(1 +r^2 -2r \cos\theta 
)^2} \sin^n \tta \, d\theta
\\&= 2r \int_0^{\pi/2} \frac{((1 + r^2)^2 - (2\cos \theta)^2){\sin^2 \tta}}{((1+r^2)^2 -  (2r\cos \theta)^2)^2} \sin^{n-2} \tta \, d\theta.
\label{spherical polars 3}
\end{align}
The integrand in \eqref{spherical polars 3} vanishes only once,   at $\Theta_r := \arccos \frac{1+r^2}{2}$.
Again the weight 
$\sin^{n-2}\theta$ is an increasing function of $\theta \in [0,\frac \pi 2]$ for $n>2$,
and suppresses the contributions from the negative region $(0,\Theta_r)$ and enhances the contributions
from the positive region $(\Theta_r,\frac\pi 2)$ relative to the case $n=2$.
 Thus $G_2(r) > 0$ for $n>2$ will follow once we have established $G_2(r)=0$ for $n=2$. For this, taking $n=2$ in \eqref{spherical polars 2} yields
\begin{align}\nonumber
{\frac{2\omega_2}{\omega_1}} \int_0^r G_2(s) ds 
&= 
- \int_0^\pi \left[\frac1{1 +s^2 -2s \cos\theta}\right]_{s=0}^r  \sin^2\theta \, d\theta
\\&= \frac \pi2 - \int_0^\pi \frac{\sin^2\theta \, d\theta}{1 +r^2 -2r \cos\theta}.
\label{cancel 1}
\end{align}
Now we change variables from $\tta$ to the angle $\phi$ between the horizontal axis and the vector from $re_1$ to $y=(\cos \tta, \sin \tta)$. Notice then $\phi > \tta$ and the last integrand above becomes $\sin^2 \phi$. By considering the triangle with vertices $\{0, y, re_1\}$, we find
\begin{align*}
 \frac{\sin(\phi - \tta)}{r} &= \frac{\sin\theta}{\sqrt{1+r^2 -2r \cos \theta}} = \sin \phi, \q \text{thus} 
\\\tta &= \phi - \sin^{-1}(r\sin\phi), \ \text{ and } 
\\  \frac{d\tta}{d\phi} &= 1 - \frac{ r\cos \phi}{\sqrt{1-r^2\sin^2 \phi}}.
\end{align*}
This yields
\begin{align}\nonumber
\int_0^\pi \frac{\sin^2\tta}{(r-\cos \tta)^2 + \sin^2 \tta} d\tta 
&= \int_0^\pi \sin^2 \phi \bigg(1 - \frac{ r\cos \phi}{\sqrt{1-r^2\sin^2 \phi}}\bigg) d\phi
\\&= \frac\pi 2
\label{cancel 2}
\end{align}
using antisymmetry of the ratio around $\phi=\frac\pi2$ again. Combining \eqref{cancel 1} with \eqref{cancel 2} gives the 
desired identity
$G_2(r)=0$ when $n=2$ and $|r|<1$ to complete the proof of the lemma. 
\end{proof}

\begin{proof}[Proof of Theorem \ref{T:main}]
For $\mu \in \cP_0(\R^n)$, define $V_\mu(x) := \int W_{\al,2}(x-y) d\mu(y)$. Then
\begin{align*}
V_\mu(x) &= \frac{1}{\al} \int |x-y|^\al d\mu(y) - \frac{1}{2}|x|^2  - \frac{1}{2}\Var(\mu),\ \text{and} \\
\nabla V_\mu(x) &= \int |x-y|^{\al - 2}(x-y) d\mu(y) - x. 
\end{align*}
For $ \al \in (\max\{4-n,2\},4)$, if $\mu$ is spherically symmetric and compactly supported,  we claim 
$f_\mu \in C_{loc}^{\lfloor \alpha\rfloor,\al-\lfloor \al \rfloor}(\R) {\cap} C^3_{loc}((0,\infty))$ (the {latter being the} space of three times continuously differentiable functions)
 and $f_\mu'''(r)>0$ for all $r>0$, 
 where 
\begin{align}
\label{f0} f_\mu(r) &:= V_\mu(re_1),  \\
\label{f1} f'_\mu(r) &= \nabla V_\mu(re_1) \cdot e_1 = \int |re_1 - y|^{\al-2}(r-y_1) d\mu(y) - r,  \\
\label{f2} {f''_\mu(r)} &= \int [(\al-2) \frac{(r -y_1)^2}{|re_1-y |^{4-\al}} + |r e_1 - y |^{\al-2} ] d\mu(y) -1, 
\end{align}
{and, if $\al \ge 3$ (or if $\mu=\sigma_1$ and $n \ge 2$) also}
\begin{align}
\label{f3} \frac{f'''_\mu(r)}{\al-2}
 &=\int \frac{r-y_1}{| re_1 -y |^{6-\al}} ( 3|re_1 - y|^2 +(\al-4) (r-y_1)^2) d\mu(y).
\end{align}

From these claims, we deduce $f_\mu(r)$ can have at most one inflection point $R$ in the range $r>0$, and must be convex on $r>R$ (since $\al>\bt$ and the compact support of $\mu$ ensures $f''_\mu(r)>0$ for large $r$).   The symmetry $f_\mu(r)=f_\mu(-r)$ and smoothness asserted imply $f'_\mu({0})=0$, 
so either $f_\mu(r)$ is strictly convex globally (in which case $R=0$) or strictly concave on $|r|<R$ and strictly convex on $r>R>0$.

It remains to establish the claims. To see $f_\mu \in {C^{\lfloor \alpha\rfloor,\al-\lfloor \al \rfloor}_{loc}}(\R)$ 
note that locally at least, the compact support of $\mu \in \cP_c^{ss}(\R^n)$ implies
$f_\mu$ is an average of uniformly $C^{\lfloor \alpha\rfloor,\al-\lfloor \al \rfloor}$ functions, 
justifying the formulas \eqref{f0}--\eqref{f3} for the first $\lfloor \al\rfloor$ derivatives. 
We now show
 $f_\mu'''(r)>0$ for all $r>0$ and also, for $2<\al<3$, that $f_\mu \in C^3_{loc}((0,\infty))$; 
 the latter has already been established for $3 \le \al$, as when $n=1$.
 We do this first  for the spherical shell $\mu=\si$ of unit radius,  before tackling the general case.
In dimension $n=1$, where $\si = \frac12(\delta_{-1} + \delta_{1})$ and $\al>3$ by hypothesis,
\[
\frac{2f_\sigma'''(r)}{(\al-1)(\al-2)}= |r-1|^{\al-4}(r-1) + |r+1|^{\al-4}(r+1)
\]
has the same sign as $r$, as desired.
To show $f_\sigma \in C^3_{loc}((0,\infty))$ for $2< \al < 3$ a more delicate argument is needed.
Note that both summands occurring in the integrand of $f_\mu'''(r)$ are dominated in absolute value by constant multiples of 
\[
|r e_1 - y|^{\al-3} \le |1- y_1^2|^{(\al-3)/2} \in L^1(\R^n, d\si).
\]
This allows $f_\sigma'''$ to be obtained by differentiating under the integral defining $f_\sigma''$ in the usual way:  approximating the derivatives by difference quotients and combining the mean value theorem with Lebesgue's dominated convergence theorem, noting $\al>2$; c.f. \cite[Theorem 2.27]{Folland99}.
Thus $f_{\sigma} \in C^3_{loc}((0,\infty))$ and its derivatives coincide with the integrals given above. 
Now $(r-y_1)^2 = |re_1 - y|^2 - 1 + y_1^2$ yields
\begin{equation}\label{f3}
\frac{f_\sigma'''(r)}{\al-2} = (\al-1) g_\al(r) + (4-\al) G_\al(r)>0,
\end{equation}
where $g_\al$ and $G_\al$ are the positive continuous functions from Lemma~\ref{positive}.

If instead $\mu=\si_R$ is the spherical shell of radius $R>0$, it follows that
$f_{\si_R}'''(r) = R^{\al-3} f_\si'''(r/R)$ is a positive continuous function of both variables $(r,R) \in (0,\infty)^2$. 
Along the $R\to 0$ boundary, Lebesgue's dominated convergence theorem shows 
\begin{align*}
 \frac{f_{\si_R}'''(r)}{(\al -2)(\al-1)} 
 &= \int  \frac{r-Ry_1}{| re_1 -Ry |^{4-\al}} \left( 1 +\frac{4-\al}{\al-1} \frac{1-y_1^2}{|re_1-Ry|^2}R^2\right) d\sigma_1(y)
\end{align*}
converges to ${\frac1{(\al-2)(\al-1)}} f'''_{\si_0}(r)=r^{{\al-}3}$. Thus $f'''_{\sigma_R}(r)$ depends continuously on $(r,R)$ in any compact rectangle $[r_0,r_1] \times [0,R_1]$ 
satisfying $r_0>0$; it admits a uniform modulus of continuity $\omega$ which depends on the rectangle.
Any spherically symmetric distribution $\mu \in \cP_c^{ss}(\Rn)$ supported on the centered ball of 
radius $R_1$ can be expressed as a weighted average of spherical shells $\sigma_R$ with $R \in[0,R_1]$. 
Therefore $f_\mu \in C^3([r_0,r_1])$ and its third derivative inherits both the desired positivity from $f'''_{\sigma_R}$ 
and the modulus of continuity $\omega$.
\end{proof}

\begin{proof}[Proof of Corollary \ref{C:abc}]
(a) {Set $f(r)=f(-r)=(W_{\al,\bt} * \mu)(r e_1)$. We first notice that the unique inflection point $R$ in Theorem \ref{T:main} is strictly positive, since if $R=0$ then $f$ is strictly convex on $\R$ and even, so the condition in (a) implies $\mu = \delta_0$, but then it is clear that $f''(0) = -1 <0$, a contradiction. Hence $R >0$. Now} since $f(r)$ is differentiable, strictly concave on $[0,R)$, strictly convex on $[R,\infty)$ and grows without bound as $r \to \infty$,  its derivative vanishes only at $r \in \{0,r_*\}$, where $r_*>R$ is its minimum.  
 Thus $\nabla (W_{\al,\bt} * \mu)(x)$
vanishes only at $|x| \in \{0,r_*\}$, so $\mu = (1-s) \sigma_{r_*} + s\delta_0$ for some $s \in [0,1]$.

(b {and e)} If the energy $\mathcal{E}_{W_{\al,\bt}}$ on a $d_\infty$-ball in {either $\cP_c(\Rn)$ or in $\cP^{ss}_c(\Rn)$} is minimized at the center $\mu$ of the ball,
then each point $x \in \spt \mu$ is a local minimum of $W_{\al,\bt}*\mu$ \cite{McCann06} \cite{BalagueCarrilloLaurentRaoul13}.
This can be thought of as a modified version of the Euler-Lagrange equation \eqref{Euler-Lagrange}, appropriate for $d_\infty$-local instead of global minima.
Then $\nabla (W_{\al,\bt} * \mu)=0$ holds $\mu$-a.e.
{If $n \ge 2$, Corollary \ref{radialmin} asserts spherical symmetry of $\mu$ after translation; the same conclusion is true by hypothesis in case (e).}
Now (a) implies 
$\mu = (1-s) \sigma_{r_*} + s\delta_0$ for some $s \in [0,1]$ and $r_*>0$.  Moreover $s=0$, since otherwise $x=0 \in \spt \mu$ is a strict local maximum of $W_{\al,\bt}*\mu$, as observed in the proof of (a).

We obtain the precise value of $r_*=R$ as follows: for a local minimum of 
\[
f_{\sigma_R}(r) = R^\al f_{\si_1}\left(\frac r R\right) + (R^\al-R^2)\left(\frac{r^2}{R^2}{+} 1 \right)
\]
to occur at $r=R$ requires 
\[
0=f'_{\sigma_R}(R) = R^{\al-1} f'_{\si_1}(1) + (R^\al-R^2)\frac{2}{R}.
\]
Since $r=0$ {does not minimize locally,} this implies 
$R =\big{(}\frac{2}{ f'_{\si_1}(1)+2} \big{)}^{\frac1{\al-2}}$.  

(c) The strict convexity of the energy shown in Corollary \ref{C:strictconvexity} below implies the minimizer $\mu$ is unique and spherically symmetric. It has compact support 
by Proposition \ref{exist}. The desired conclusion now follows from (b).

(d) In the remaining case $(\al,\bt,n)=(3,2,1)$, at least one minimizer is a spherical shell by continuity as $\al \searrow 3$, and a priori uniqueness of the minimizer completes the proof.
\end{proof}

\begin{remark}[Optimizers on the real line]
\label{1D references}
(i) Kang, Kim, Lim and Seo \cite[Theorem 2 (2),(6)]{KKLS21} showed for $(n,\bt)=(1,2)$ and $\al \ge 3$ that 
$\mu_*:= \frac{1}{2}(\delta_{- \frac{1}{2}} + \delta_{ \frac{1}{2}})$ is a strict local minimizer with respect to the $d_\infty$ metric on $\cP_{0}(\R)$. Our Corollary \ref{C:abc}  improves this by showing that $\mu_*$ is {the unique compactly supported $d_\infty$-local minimizer with even symmetry for all $\al > 3$, and} the unique global minimizer for all $\al \ge 3$. (ii) \cite[Theorem 2 (4)]{KKLS21}  showed in the range $(n,\bt)=(1,2)$ and $2< \al <3$
excluded by Corollary \ref{C:abc}, that $\mu_*$ is not a $d_\infty$-local minimizer, hence not a global minimizer.   
The global minimizer in this range has since been identified by Frank \cite{Frank22} {in response to \cite{DaviesLimMcCann21+b}}. 
(iii) \cite[Example 1, Theorem 2 (1)]{KKLS21} jointly show that e.g. if $(\al,\bt,n) =(2.4,2.1,1)$, $ \mu_*$ is a $d_\infty$-strict local minimizer but not a global minimizer. 
\end{remark}

\section{Asymptotic vs. Lyapunov stability of spherical shells}

Finally, let us clarify the sense in which these energy minimizers represent stable solutions to the aggregation equation \eqref{aggregation}.
Near a fixed point of a dynamical system,  there are several possible notions of nonlinear stability.  {\em Asymptotic stability}, 
requires the fixed point to attract all solutions in its neighbourhood, i.e. to form what is called an {\em attractor}.  {\em Lyapunov stability} is a weaker notion,  which merely requires 
that no point starting sufficiently close to the fixed point strays too far away: for any $\ep$-ball around the fixed point $x$ there should be an open ball $B_\delta(x)$ of initial conditions whose future trajectories remain in $B_\epsilon(x)$.  
Both notions admit obvious extensions to families of fixed points.  Moreover,
both 
notions are sensitive to the metric (or topology) in which closeness is measured, and to the class of initial conditions permitted.

For the Kantorovich-Rubinstein-Wasserstein family of distances 
\eqref{KRW metric},
energy minimizers need not form an {asymptotically} stable family:
in two-dimensions our next example shows that the
discrete ring solutions (whose stability was investigated linearly by Bertozzi, Kolokolnikov, Sun, Uminsky and von Brecht 
\cite{KolokolnikovSunUminskyBertozzi11} \cite{BertozziKolokolnikovSunUminskyVonBrecht15}
and nonlinearly by Simione \cite{Simione14})
provide non-minimizing steady states arbitrarily
close to the energy minimizing spherical shell. 

\begin{example}[Steady state periodic rings of many point masses]\label{E:ring}
Let $\al >  \bt>1$ and $n=2$. Assume for some $R>0$ that the spherical shell 
$\sigma_R$ is a steady-state for \eqref{aggregation},
so that $V_{\sigma_R} (x) = W_{\al,\bt} * \sigma_R (x)$ satisfies $\nabla V_{\sigma_R} = 0 \mbox{ on } \spt(\sigma_R)$. 
Then for any $\ep >0$, there exists a steady-state $\omega$ which takes the form of a discrete ring of uniformly spaced identical point masses \eqref{k-ring}
and satisfies $d_\infty(\omega, \sigma_R) < \ep$. 
\end{example}

\begin{proof} Let $g(r)= \nabla V_{\sigma_r}(re_1) \cdot e_1$. We claim $g'(r)>0$ whenever $g(r)=0$ and $r>0$. To see this, denote $y_1=e_1 \cdot y$  and compute
\begin{align*}
g(r)&=\int [ |re_1 - y|^{\al-2}(r-y_1) - |re_1 - y|^{\bt-2}(r-y_1) ] d\sigma_r(y)  \\
&=\int [|re_1 - ry|^{\al-2}(r-ry_1) - |re_1 - ry|^{\bt-2}(r-ry_1)] d\sigma_1(y) \\
&= r^{\al-1}\int |e_1 - y|^{\al-2}(1-y_1) d\sigma_1(y) -  r^{\bt-1}\int |e_1 - y|^{\bt-2}(1-y_1) d\sigma_1(y) \\
&=c_\al r^{\al-1} - c_\bt r^{\bt-1}
\end{align*}
where $c_\al = \int |e_1 - y|^{\al-2}(1-y_1) d\sigma_1(y)$ and $c_\bt$ similarly. 
Now $g(R)=0$ gives $c_\bt R^{\bt-2} = c_\al R^{\al-2}$, thus $g'(R) = c_\al R^{\al-2} (\al - \bt) > 0$, as claimed. Let 
\be\label{k-ring}
\omega_{k,r} = \frac1k\sum_{m=1}^k \delta_{re^{2\pi i m/k}}
\ee 
denote the discrete ring supported on $k$ points spaced uniformly around the circle of radius $r$. 
We assume $re_1 \in \spt (\omega_{k,r})$.
  Now $g(R)=0$ by assumption, thus given $\eta >0$, the claim allows us to find $R^+ \in (R, R+\eta)$, $R^- \in (R-\eta, R)$ such that $g(R^+) >0$, $g(R^-) < 0$. Hence by approximation, for all large enough $k$ we have
\be
 \nabla V_{\omega_{k,R^+}}(R^+e_1) \cdot e_1 >0, \q  \nabla V_{\omega_{k,R^-}}(R^-e_1) \cdot e_1 <0. \nn
\ee
By continuity, there exists $R^* \in (R^-, R^+)$ so that $ \nabla V_{\omega_{k,R^*}}(R^* e_1) \cdot e_1 =0$. By symmetry $ \nabla V_{\omega_{k,R^*}}(R^* e_1)$ is parallel to $e_1$, hence $ \nabla V_{\omega_{k,R^*}}(R^* e_1) = 0$. By symmetry again $ \nabla V_{\omega_{k,R^*}}$ vanishes on $\spt(\omega_{k,R^*})$, that is $\omega_{k,R^*}$ is a steady-state. Now $d_\infty(\sigma_r, \omega_{k,r}) \to 0$ {uniformly on $r \in [R-\epsilon,R+\epsilon]$} as $k \to \infty$. 
Thus for $\epsilon>0$, we choose $\eta  < \epsilon/2$, $k$ large enough and $R^* \in (R-\eta,R+\eta)$ as above to ensure
 \[
 d_\infty(\sigma_R,\omega_{k,R^*}) {\le d_\infty(\sigma_R,\sigma_{R^*}) + d_\infty(\sigma_{R^*},\omega_{k,R^*})}
 <\ep
 \]
  as desired. 
\end{proof}

\begin{remark}[Minimizers need not be asymptotically stable]
By the Euler-Lagrange equation \eqref{Euler-Lagrange}, this example implies there exists a steady-state discrete ring which is $d_p$-arbitrarily close to the spherical shell minimizer given in Corollary \ref{C:abc} for each $p\in[1,\infty]$.  This nearby accumulation of non-minimizing steady-states shows that even though the energy \eqref{energy} is a Lyapunov function for the aggregation dynamics \eqref{aggregation},  the spherical shells which minimize it are not asymptotically stable.
 This is in sharp contradistinction to the results of Simione{. In particular, \cite[Theorems 16 and 25]{Simione14} asserts that if a finitely supported steady state is fully linearly stable, then it is both $d_\infty$-Lyapunov stable, and
 $(d_\infty; d_2)$-asymptotically stable, in the sense that probability measures which start $d_\infty$-close to the steady state contract to the steady state in the $d_2$-metric under the aggregation dynamics \eqref{aggregation}. In view of Corollary \ref{radialmin} however,  the discrete ring solutions of the previous example are not fully linearly stable when $\bt=2<\al<4$.
 Likewise, our results differ from those} of 
 Balagu\'e et al \cite{BalagueCarrilloLaurentRaoul13}, which show local $d_p$-asymptotic stability of our spherical shell
 for some $p\ge 1$ in the more restricted class  $\cP_c^{ss}(\Rn)$
  of spherically symmetric initial data. 
\end{remark}

 Instead, the minimizing family is {\em $d_\al$-Lyapunov} stable,  in the sense that the evolution stays as $d_\al$-close as we please to 
the minimizing family  if it starts $d_\al$-close enough to it. 
Recall the following variant of a 
well-known lemma,  {in which {\em coercivity of} $E:X \longrightarrow \R$ means $E^{-1}((-\infty,h])$ is assumed to be compact for each $h \in \R$.}

\begin{lemma}[Lyapunov stability]\label{L:Lyapunov}
If $E:X \longrightarrow \R$ is a continuous coercive function on a metric space $(X,d)$ and $Y\subset X$, then for each $\ep>0$ there exist $\de>0$ and $h\in \R$ such that
\be\label{energy levels}
(\argmin_X E)^\de \subset E^{-1}((-\infty,h)) \subset (\argmin_X E)^\ep,
\ee
where
\[
Y^\ep := \{ x \in X \mid d(x,Y) := \inf_{y \in Y} d(x,y) < \ep\}.
\]
\end{lemma}

\begin{proof}
Continuity and coercivity imply $E$ attains its minimum on $X$. Let $A=\argmin_X E$ denote the set of minimizers and 
$e'= \min_X E$.
Given $\ep>0$, to derive a contradiction suppose no $h>e'$ satisfies the second inclusion \eqref{energy levels}.
Then for each $k \in \bbN$ there exists $x_k \in X\setminus A^\ep$ with $E(x_k)<e' + 1/k$.  Coercivity yields an accumulation point
$x_\infty$ of $\{x_k\}_{k\in \bbN}$,  which must lie in the closed set $X \setminus A^\ep$.  Continuity of $E$ yields $E(x_\infty)= e'$,
hence $x \in A$ --- the desired contradiction.  We now provide $\de$ satisfying the first inclusion:  since continuity of $E$ yields an open set $E^{-1}((-\infty,h))$ containing
the compact set $A$,  the distance of $A$ to the closed set $E^{-1}([h,\infty))$ is positive; taking $\delta$ to be this positive distance  establishes the lemma.
\end{proof}

To apply the lemma,  we take $X$ to consist of the centered measures with finite $\al$-th moments:
\begin{equation}\label{P_a}
\cP_{\al,0}(\Rn) := \{\mu \in \cP_0(\Rn) \mid \int_{\Rn} |x|^\al d\mu(x)<\infty  \}. 
\end{equation}

\begin{corollary} [$d_\al$-Lyapunov stability]
For $0< \bt < \al <\infty$  and $\al \ge 1$, taking $E=\cE_{W_{\al,\bt}}$ and $(X,d)=(\cP_{\al,0}(\Rn),d_\al)$, 
the preceding lemma ensures that any curve $(\mu(t))_{t\ge 0} \in X$ starting within distance $\delta>0$ of an energy minimizer
remains within distance $\ep>0$ of an energy minimizer as long as $E(\mu(t)) \le E(\mu(0))$ for all $t \ge 0$.
\end{corollary}

\begin{proof}
Let $d(x,y)=|x-y|$. Jensen's inequality, $|| d ||_{L^\bt(\mu \otimes \mu)} \le || d ||_{L^\al(\mu \otimes \mu)}$, converts the energy bound
\begin{align*}
h &\ge  \cE_{W_{\al,\bt}}(\mu)
\\ &= \frac1\al \|d\|^\al_{L^\al(\mu \otimes \mu)} - \frac1\bt \|d\|^\bt_{L^\bt(\mu \otimes \mu)}
\\ &\ge  \frac1\al \|d\|^\al_{L^\al(\mu \otimes \mu)} - \frac1\bt \|d\|^\bt_{L^\al(\mu \otimes \mu)}
\end{align*}
into a bound on $\|d\|_{L^\al(\mu \otimes \mu)}$.  For $\mu \in \cP_0(\Rn)$,  a second application of Jensen's inequality
\begin{align*}
\|d\|^\al_{L^\al(\mu \otimes \mu)}
 &= \iint_{\Rn\times \Rn} |x-y|^\al d\mu(x) d\mu(y)
 \\& \ge \int_\Rn |x|^\al d\mu(x) 
 \\&= d_\al(\mu, \delta_0)^\al
\end{align*}
then shows $E^{-1}((-\infty,h])$ is $d_\al$-bounded.
Now \cite[Theorem 7.12]{Villani03} shows $d_\al$-continuity of $\cE_{W_{\al,\bt}}$ on \eqref{P_a},
so $E^{-1}((-\infty,h])$ is also $d_\al$-closed. Finally, \cite[Theorem 2.7]{AmbrosioGigli13} asserts closed and bounded subsets of $\cP_2(\Rn)$ are $d_2$-compact,
but for $\al \in [1,\infty)$ the same proof yields $d_\al$-compactness of closed bounded subsets of $\cP_\al(\Rn)$ with respect to the distance 
\eqref{KRW metric}. This establishes the desired coercivity of $E$ on $(X,d)$.
\end{proof}

\begin{remark} [Lyapunov stability of aggregation near energy minima]
Since the aggregation equation \eqref{aggregation} preserves center of mass without increasing the energy \eqref{energy}, 
the last corollary asserts the desired Lyapunov stability result.
To obtain this stability, our distance $d_\al$ is adapted to match the largest exponent in the interaction potential $W=W_{\al,\bt}$.
Note that we need not specify a solution concept for the dynamics, so long as it preserves (sign, mass, center of mass) and dissipates energy. 
\end{remark}

\begin{appendices}
\section{Appendix: Strict convex/concavity of $\mathcal{E}_{W_\al}$ for $\al \in (2,4) \cup (0,2)$}

\label{S:Strict convexity}

We end by recalling that the strict convexity proved by Lopes \cite{L19} for the functional $\mathcal{E}_{W_\alpha}$ 
extends to singular measures.   After posting this manuscript on the arXiv,  we learned from Rupert Frank that this extension was previously established in the course of proving uniqueness of energy minimizer up to transition for an interacting gas model satisfying a polytropic equation of state: Theorem 27 of \cite{CarrilloDelgadinoDolbeaultFrankHoffman19}.
 We nevertheless include our own proof, which differs in some ways from that of \cite{CarrilloDelgadinoDolbeaultFrankHoffman19}.

Convexity of the quadratic form \eqref{energy} is equivalent to non-negative
definiteness of its kernel $W$. For the kernels $\{W_\alpha\}_{0<\al<4}$,
Lopes \cite{L19} explored this sign definiteness 
using Fourier transforms.  Although
he was only interested in the action of such kernels on 
probability densities,  we now show his considerations extend also to singular probability measures, while retaining strict convexity. The Fourier transform of
a signed measure is defined by
\be\label{Fourier}
\hat \rho(\xi) :=\int_{\R^n} e^{-2\pi i \xi \cdot x} d\rho(x). 
\ee
For bounded densities $\rho$ and kernels $W$ of compact support,  Plancherel's formula 
\begin{equation}\label{Parseval}
\langle\rho, W*\rho\rangle_{L^2(\R^n)} 
= \langle \hat \rho, \hat W \hat \rho\rangle_{L^2(\Rn)}
\end{equation}
shows the sign of $\hat W$ determines the sign-definiteness of the quadratic form.  For singular measures and long range, unbounded kernels,  things are potentially more delicate.
From Euler's product representation
\begin{equation}\label{Euler product}
\Gamma(z) = \frac{1}{z} \prod_{n=1}^\infty \frac{(1+\frac1 n )^z}{1+\frac z n}
\end{equation}
recall $\Gamma(z)$ is analytic except at the negative integers, where it has simple poles (and where its restriction to the real axis therefore changes signs). {Let $\mathcal{P}_c(\R^n)$ denote the set of compactly supported probability measures.}
We extend Lopes' result with the following
analog of \eqref{Parseval}:

\begin{proposition}[Sign of kernel action on centered neutral measures]
\label{P:Lopes Formula}
If $\rho=\mu-\nu$ is the difference of $\mu,\nu\in \mathcal{P}_{0}(\R^n) \cap \mathcal{P}_{c}(\R^n)$, then 
\begin{align}\label{Fal}
F_\alpha(\rho)
:&=\iint_{\R^n\times \R^n} |x-y|^\alpha d\rho(x)d\rho(y)
\\&= C(\alpha)\int_{\R^n}|\xi|^{-\alpha-n}|\hat{\rho}(\xi)|^2d\xi =:\tilde F_\alpha(\hat\rho)
\end{align}
for each $\alpha\in(0,2)\cup(2,4)$, where $C(\alpha):=2^{\alpha+n/2}\frac{\Gamma((\alpha+n)/2)}{\Gamma(-\alpha/2)}$.
\end{proposition}

We shall derive this from Lopes' results using approximation.
\begin{proof}
Lopes shows this result on $\mathcal{P}_0(\R^n)\cap \mathcal{P}_c(\R^n)\cap L^1(\R^n)$, for the given range of $\alpha$. We shall extend it to $\mathcal{P}_0(\R^n)\cap \mathcal{P}_c(\R^n)$. Let $\mu,\nu \in \mathcal{P}_{0}(\R^n) \cap \mathcal{P}_{c}(\R^n)$. Choose a smooth radial density $\varphi \in \mathcal{P}_0(\R^n)\cap \mathcal{P}_c(\R^n)$ supported in a unit ball, and consider the mollified measures $(\mu_\ep)_{\ep > 0}$, $(\nu_\ep)_{\ep > 0}$ defined by $d\mu_\ep(x)=(\varphi_{\ep}*\mu)(x)dx$, where $\varphi_{\ep}(x):=\frac{1}{\ep^n}\varphi(\frac{x}{\ep})$ and $\varphi$ is the usual smooth probability density compactly supported on the unit ball. It is then easy to check $\mu_{\ep}, \nu_\ep \in \mathcal{P}_0(\R^n)\cap C_c^\infty(\R^n)$, and the functions $\mu_\ep, \nu_\ep$ are uniformly supported in a ball of radius $R$ for all $\ep \in (0,1]$. Moreover,  $d_\infty(\mu_\ep, \mu), d_\infty(\nu_\ep, \nu) \le \ep$ hence $\mu_\ep \to \mu$, $\nu_\ep \to \nu$ as $\ep \to 0$ in $\infty$-Wasserstein distance \eqref{KRW metric} on $\cP(\Rn)$. Thanks to this $d_\infty$-convergence, we obtain $F_\alpha(\rho) = \lim_{\ep \to 0} F_\alpha(\rho_\ep)$ for any $\al > 0$.

We next show that $\tilde F_\alpha(\hat\rho_\ep)\to\tilde F_\alpha(\hat\rho)$. We split the integral as follows:
\[
\frac{\tilde F_\alpha(\hat{\rho}_\ep)}{C(\alpha)}=\int_{B_1(0)}|\xi|^{-\al-n}|\hat{\rho_\ep}(\xi)|^2d\xi+\int_{\R^n\setminus B_1(0)}|\xi|^{-\al-n}|\hat{\rho_\ep}(\xi)|^2d\xi.
\]
We show each integral converges as $\ep \to 0$. By Schwartz's Paley-Wiener theorem \cite{Schwartz52} 
 for distributions, $\hat{\rho}_\ep$ is analytic for all $\ep \ge 0$. Since vanishing zeroth and first moments imply $\hat{\rho}_\ep(0)=\int d\rho_\ep(x)=0$
and $(\nabla_\xi\hat{\rho}_\ep)(0)=0$, we find $\frac{\hat{\rho_\ep}(\xi)}{|\xi|^2}$ is also analytic. Then the power series expansion at the origin implies  $\frac{\hat{\rho_\ep}(\xi)}{|\xi|^2}$ is uniformly bounded in $B_1(0)$ for all $\ep \in [0,1]$, since all mixed partial derivatives of order $k$ of $\hat{\mu_\ep}$, $\hat{\nu_\ep}$  at $0$ are bounded by $R^k$ by the basic property of Fourier transforms \eqref{Fourier}. Hence

\begin{align*}
\int_{B_1(0)}|\xi|^{-\al-n}|\hat{\rho_\ep}(\xi)|^2d\xi
=&\int_{B_1(0)}|\xi|^{-\al-n+4}\left(\frac{|\hat{\rho_\ep}(\xi)|}{|\xi|^2}\right)^2d\xi
\\\xrightarrow{\ep\to 0} & \int_{B_1(0)}|\xi|^{-\al-n+4}\left(\frac{|\hat{\rho}(\xi)|}{|\xi|^2}\right)^2d\xi
\\=&\int_{B_1(0)}|\xi|^{-\al-n}|\hat{\rho}(\xi)|^2d\xi
\end{align*}
since $\al <4$ and pointwise convergence of $\hat{\rho_\ep}$ to $\hat{\rho}$, proving the convergence of the first integral by Lebesgue Dominated Convergence Theorem. Next, since $|\hat\rho_{\ep}(\xi)|\leq 2$ for any $\xi$ and  $|\xi|^{-\alpha-n}\in L^1(\R^n\setminus B_1(0))$, we similarly have 
\[
\int_{\R^n\setminus B_1(0)}|\xi|^{-\al-n}|\hat{\rho_\ep}(\xi)|^2d\xi 
\xrightarrow{\ep\to 0} \int_{\R^n\setminus B_1(0)}|\xi|^{-\al-n}|\hat{\rho}(\xi)|^2d\xi.
\]
Summing up we deduce $\lim_{\ep\to 0} \tilde{F}_\al(\hat\rho_\ep)=\tilde{F}_\al(\hat\rho)$, and thereby obtain
\[
F_\al(\rho)=\lim_{\ep\to 0} F_\al(\rho_\ep)=\lim_{\ep\to 0} \tilde{F}_\al(\hat\rho_\ep)=\tilde{F}_\al(\hat\rho)
\]
where the second equality is due to Lopes \cite{L19}.
\end{proof}

\begin{corollary}[Energetic convexity for singular measures]
\label{C:strictconvexity} 
On $\cP_0(\R^n) \cap \cP_c(\R^n)$, $\mathcal{E}_{\Wa}$ is strictly convex if $2 < \al < 4$, and is strictly concave if $0 < \al < 2$. In addition, $\mathcal{E}_{\Wa}$ is convex if $\al=4$, and is linear if $\al=2$.
\end{corollary}

\begin{proof} Let $\mu_0,\mu_1\in \cP_0(\R^n)\cap\cP_c(\R^n)$, $\rho= \mu_1-\mu_0$ and let $a(t)=\mathcal{E}_{W_\alpha}((1-t)\mu_0+t\mu_1)$ denote the energy along the line segment between $\mu_0$ and $\mu_1$. Given that we will be interested in questions of convexity, we note that $a''(t)=\mathcal{E}_{W_\alpha}(\rho)$, so convexity of $\mathcal{E}_{W_\alpha}$ depends exclusively on the sign of $F_\al(\rho) = 2\al \mathcal{E}_{W_\alpha}(\rho) $. 

We first address the $\alpha\in(0,2)$ and the $\alpha\in (2,4)$ cases. In either of these cases, we apply the formula from Proposition \ref{P:Lopes Formula} to see that 
\[
F_\al(\rho)=C(\alpha)\int_{\R^n}|\xi|^{-n-\alpha}|\hat\rho(\xi)|^2 d\xi.
\]
Since $|\xi|^{-n-\alpha}$ is strictly positive on $\R^n\setminus\{0\}$, this integral vanishes if  and only if $\hat\rho=0$ on $\R^n$ (recall $\hat\rho$ is continuous). However, by the injectivity of the Fourier-Stieltjes transform, this only happens if $\mu_0=\mu_1$, so we can conclude that, unless $\mu_0=\mu_1$, $\int |\xi|^{-n-\alpha}|\hat\rho(\xi)|^2d\xi>0$, and hence, $F(\rho)$ will take the sign of $C(\alpha)$ if $\mu_0\neq \mu_1$. Now $C(\alpha)<0$ for $\alpha\in(0,2)$ and $C(\alpha)>0$ for $\alpha\in (2,4)$
according to \eqref{Euler product}. This yields strict concavity in the former case and strict convexity in the latter. 

If $\al=2$, it is easily seen $\mathcal{E}_{W_2}(\mu) = \frac12 \int |x|^2 d\mu(x)$ 
hence depends linearly instead of quadratically on $\mu \in \cP_0(\Rn)$, while $\mathcal{E}_{W_4}=\lim_{\al \nearrow 4} \mathcal{E}_{W_\al}$
implies (not necessarily strict) convexity of $\mathcal{E}_{W_4}$.
\end{proof}

\begin{corollary}[{Spherical symmetry of $d_\infty$-local} energy minimizers]
\label{radialmin}
If $(\al,\bt) \in [2,4] \times (0,2]\setminus \{(4,2),(2,2)\}$ {and $\mu$ minimizes $\mathcal{E}_{\Wab}$ on an open $d_\infty$-ball in $\cP_c(\Rn)$ then, after translation, $\mu$} is spherically symmetric {if $n \ge2$}. {Apart from translations, $\mathcal{E}_{\Wab}$ has a unique global mininimum on $\cP(\Rn)$ for all $n \ge 1$.}  
\end{corollary}

\begin{proof} Corollary \ref{C:strictconvexity} shows $\mathcal{E}_{W_{\alpha,\beta}}$ is strictly convex, so a standard convexity argument shows that, if $\mu, \nu \in \cP_0(\R^n)\cap\cP_c(\R^n)$ are distinct measures with the same energy, then $\frac{\mu+\nu}{2}$ will have strictly lower interaction energy than either.  {Since Lemma \ref{exist} shows global energy minimizers have bounded support,  there can only be one such minimizer centered at the origin.  
On the other hand,  if $\mu$ minimizes $\cP_c(\Rn)$
on a $d_\infty$-ball of radius $\epsilon$,  we may translate it to have center of mass at the origin. Any slight rotation $\nu:=R\mu$ has the same energy as $\mu$. For a small enough rotation,  $\frac12 (\mu +\nu)$ lies within $d_\infty$ distance
$\epsilon$ of $\mu$ and has strictly lower energy --- producing a contradiction unless $\mu=R\mu$ {(or $n=1$, in which case the only small rotation is trivial)}.  Thus $\mu$ is invariant under all small (and hence large) rotations if $n\ge2$:} i.e.~$\mu$ has the desired spherical symmetry. 
\end{proof}
\end{appendices}


\begin{thebibliography}{1}
 
 \bibitem{AlbiBalagueCarrilloVonBrecht14}
G.~Albi, D.~Balagu\'{e}, J.~A. Carrillo, and J.~von Brecht.
\newblock Stability analysis of flock and mill rings for second order models in
  swarming.
\newblock {\em SIAM J. Appl. Math.}, {\bf 74} (2014) 794--818.

\bibitem{AmbrosioGigli13}
Luigi Ambrosio and Nicola Gigli.
\newblock A user's guide to optimal transport. In {\em Modelling and optimisation of flows on networks}, 1--155,
Springer, Heidelberg, 2013.

\bibitem{AmbrosioGigliSavare05}
Luigi Ambrosio, Nicola Gigli and Giuseppe Savar\' e.
\newblock Gradient flows in metric spaces and in the space of probability measures. 
\newblock 
Birkhuser Verlag, Basel, 2005. 

\bibitem{BalagueBarbaroCarrilloVolkin20}
Daniel Balagu\'e Guardia, Alethea Barbaro, Jose A. Carrillo, and Robert Volkin.
\newblock Analysis of spherical shell solutions for the radially symmetric aggregation equation. 
\newblock {\em SIAM J. Appl. Dyn. Syst.}, {\bf 19} (2020), no. 4, 2628--2657. 
  
\bibitem{BalagueCarrilloLaurentRaoul13}
D.~Balagu\'{e}, J.~A. Carrillo, T.~Laurent, and G.~Raoul.
\newblock Dimensionality of local minimizers of the interaction energy.
\newblock {\em Arch. Ration. Mech. Anal.}, {\bf 209} (2013) 1055--1088.

\bibitem{BalagueCarrilloLaurentRaoul13N}
D.~Balagu\'{e}, J.~A. Carrillo, T.~Laurent, and G.~Raoul.
\newblock Nonlocal interactions by repulsive-attractive potentials: radial ins/stability. 
\newblock {\em Phys. D} {\bf 260} (2013), 5--25.

\bibitem{BenedettoCagliotiPulvirenti97}
D.~Benedetto, E.~Caglioti, E., and M.~Pulvirenti. 
\newblock A kinetic equation for granular media. 
\newblock {\em RAIRO Mod\'el. Math. Anal. Num\'er.} {\bf 31} (1997), no. 5, 615--641.

\bibitem{BertozziKolokolnikovSunUminskyVonBrecht15}
Andrea~L. Bertozzi, Theodore Kolokolnikov, Hui Sun, David Uminsky, and James
  von Brecht.
\newblock Ring patterns and their bifurcations in a nonlocal model of
  biological swarms.
\newblock {\em Commun. Math. Sci.}, {\bf 13} (2015) 955--985.

\bibitem{BlanchetCarlier14}
Adrien Blanchet and Guillaume Carlier.
\newblock From {N}ash to {C}ournot-{N}ash equilibria via the
  {M}onge-{K}antorovich problem.
\newblock {\em Philos. Trans. R. Soc. Lond. Ser. A Math. Phys. Eng. Sci.},
  {\bf 372:} 20130398 (2014) 11.

\bibitem{Breder54}
C.M.~Breder, Jr.
\newblock Equations descriptive of fish schools and other animal aggregations.
\newblock {\em Ecology}, {\bf 35} (1954) No. 3, 361--370.

\bibitem{BurchardChoksiTopaloglu18}
Almut Burchard, Rustum Choksi, and Ihsan Topaloglu. 
\newblock Nonlocal shape optimization via interactions of attractive and repulsive potentials.
\newblock {\em Indiana Univ. Math. J.} {\bf 67} (2018), no. 1, 375--395.

\bibitem{BurchardChoksiHess-Childs20}
Almut Burchard, Rustum Choksi, and Elias Hess-Childs. 
\newblock On the strong attraction limit for a class of nonlocal interaction energies. 
\newblock {\em Nonlinear Analysis}, Volume 198, September 2020, 111844  

\bibitem{CarrilloDelgadinoDolbeaultFrankHoffman19}
J. A. Carrillo, M. G. Delgadino, J. Dolbeault, R. L. Frank, F. Hoffmann.
\newblock Reverse Hardy-Littlewood-Sobolev inequalities. 
\newblock {\em J. Math. Pures Appl.} (9) {\bf 132} (2019), 133-165.

\bibitem{CarrilloFigalliPatacchini17}
J.A.~Carrillo, A.~Figalli, and F.~S. Patacchini.
\newblock Geometry of minimizers for the interaction energy with mildly
  repulsive potentials.
\newblock {\em Ann. Inst. H. Poincar\'{e} Anal. Non Lin\'{e}aire},
  {34} (2017), 1299--1308.

\bibitem{CarrilloHittmeirVolzoneYao19}
J.A.~Carrillo, S.~Hittmeir, B.~Volzone, and Y.~Yao.
\newblock Nonlinear aggregation-diffusion equations: radial symmetry and long time asymptotics. 
\newblock {\em Invent. Math.} {\bf 218} (2019), no. 3, 889--977.

\bibitem{CarrilloHuang17}
Jos\'e A. Carrillo and Yanghong Huang.
\newblock Explicit equilibrium solutions for the aggregation equation with power-law potentials.
\newblock {\em Kinet. Relat. Models} {\bf 10} (2017), no. 1, 171--192.

\bibitem{CarrilloMcCannVillani03}
Jos\'e~A. {Carrillo, Robert~J. McCann, and C\'edric Villani}.
\newblock Kinetic equilibration rates for granular media and related equations:
  entropy dissipation and mass transportation estimates.
\newblock {\em Revista Mat. Iberoamericana}, {\bf 19} (2003) 1--48.

\bibitem{CarrilloMcCannVillani06}
Jos\'e A. Carrillo, Robert J. McCann, and C\'edric Villani.
\newblock Contractions in the 2-Wasserstein length space and thermalization of granular media. 
\newblock {\em Arch. Ration. Mech. Anal.} {\bf 179} (2006), no. 2, 217--263.     

\bibitem{CarrilloShu21+}
Jos\'e A. Carrillo and Ruiwen Shu.
\newblock From radial symmetry to fractal behavior of aggregation equilibria for repulsive-attractive potentials.
\newblock Preprint at {\tt arXiv:2107.05079}

\bibitem{ChoksiFetecauTopaloglu15}
Rustum Choksi, Razvan~C. Fetecau, and Ihsan Topaloglu.
\newblock On minimizers of interaction functionals with competing attractive
  and repulsive potentials.
\newblock {\em Ann. Inst. H. Poincar\'{e} Anal. Non Lin\'{e}aire},
  {\bf 32} (2015) 1283--1305.

\bibitem{CraigBertozzi16}
Katy Craig and Andrea~L. Bertozzi.
\newblock A blob method for the aggregation equation.
\newblock {\em Math. Comp.}, {\bf 85} (2016) 1681--1717.

\bibitem{DaviesLimMcCann21+b}
Cameron Davies, Tongseok Lim and Robert J. McCann.
 \newblock Classifying minimum energy states for interacting particles: regular simplices.
 \newblock Preprint at {\tt arXiv:2109.07091} 

\bibitem{DelgadinoYanYao20}
Matias G. Delgadino,  Xukai Yan and  Yao Yao.
\newblock {Uniqueness and Nonuniqueness of Steady States of Aggregation-Diffusion Equations}.
\newblock {\em Communications on Pure and Applied Mathematics}, 2020. https://doi.org/10.1002/cpa.21950

\bibitem{Dong11}
Hongjie Dong.
\newblock The aggregation equation with power-law kernels: ill-posedness, mass concentration and similarity solutions. 
\newblock {\em Comm. Math. Phys.} {\bf 304} (2011), no. 3, 649--664.


\bibitem{FellnerRaoul11}
Klemens Fellner and Ga\"{e}l Raoul.
\newblock Stability of stationary states of non-local equations with singular interaction potentials. 
\newblock {\em Math. Comput. Modelling} {\bf 53} (2011), no. 7-8, 1436--1450.

\bibitem{FetecauHuang13}
R.C.~Fetecau and Y.~Huang.
\newblock Equilibria of biological aggregations with nonlocal repulsive-attractive interactions.
\newblock {\em Phys. D} {\bf 260} (2013), 49--64.

\bibitem{FetecauHuangKolokolnikov11}
R.C.~Fetecau, Y.~Huang, and T.~Kolokolnikov. 
\newblock Swarm dynamics and equilibria for a nonlocal aggregation model. 
\newblock {\em Nonlinearity} {\bf 24} (2011), no. 10, 2681--2716.

\bibitem{Folland99}
Gerald B.~Folland.
\newblock {Real Analysis.  Modern Techniques and Their Applications, 2nd Edition.}
 \newblock New York: John Wiley \& Sons, 1999.


\bibitem{Frank22}
Rupert L.~Frank.
\newblock Minimizers for a one-dimensional interaction energy. 
\newblock {\em Nonlinear Anal.} {\bf 216} (2022) 112691.


\bibitem{FrankLieb18}
Rupert L.  Frank and Elliott H. Lieb.
\newblock A ``liquid-solid'' phase transition in a simple model for swarming, based on the ``no flat-spots'' theorem for subharmonic functions.
\newblock {\em Indiana Univ. Math. J.}  {\bf 67} (2018), no. 4, 1547--1569., 

\bibitem{FrankLieb19+}
Rupert L.  Frank and Elliott H. Lieb.
\newblock Proof of spherical flocking based on quantitative rearrangement inequalities. 
\newblock {\em Ann. Sc. Norm. Super. Pisa Cl. Sci. (5).} {\bf 22} (2021) 1241--1263. 

\bibitem{HolmPutkaradze06}
Darryl~D. Holm and Vakhtang Putkaradze.
\newblock Formation of clumps and patches in self-aggregation of finite-size
  particles.
\newblock {\em Phys. D}, {\bf 220} (2006) 183--196.

\bibitem{KKLS21}
K.~Kang, H.~K.~Kim, T.~Lim and G.~Seo.
 \newblock Uniqueness and characterization of local minimizers for the interaction energy with mildly repulsive potentials.
  \newblock {\em Calc. Var. Partial Differential Equations} {\bf 60} (2021), no. 1, Paper No. 15, 17 pp. https://doi.org/10.1007/s00526-020-01882-7
  
\bibitem{KKS19}
K.~Kang, H.~K.~Kim, and G.~Seo.
 \newblock Cardinality estimation of support of the global minimizer for the interaction energy with mildly repulsive potentials.
 \newblock {\em Physica D}, 
Volume 399, 1 December 2019, 51--57. https://doi.org/10.1016/j.physd.2019.04.004

\bibitem{KellerSegel70}
Evelyn F. Keller and Lee A. Segel.
\newblock Initiation of slime mold aggregation viewed as an instability.
\newblock {\em J. Theoret. Biol.} {\bf 26} (1970), no. 3, 399--415.


\bibitem{KolokolnikovSunUminskyBertozzi11}
Theodore Kolokolnikov, Hui Sun, David Uminsky, and Andrea Bertozzi.
\newblock Stability of ring patterns arising from two-dimensional particle
  interactions.
\newblock {\em Phys. Rev. E} {\bf 84} (1) 015203 (2011).


 \bibitem{L19}
O. Lopes.
 \newblock  Uniqueness and radial symmetry of minimizers for a nonlocal
variational problem.
 \newblock {\em Comm. Pure. Appl. Anal.} 18 (2019) 2265--2282.
 
   \bibitem{LimMcCann19p}
 Tongseok~Lim and Robert J. McCann.
  \newblock Geometrical bounds for the variance and recentered moments.
\newblock {\em Math. Oper. Res.} {\bf 47} (2022) 286-296. https://doi.org/10.1287/moor.2021.1125
 
 \bibitem{LimMcCann21}
 Tongseok Lim and Robert J. McCann.
 \newblock Isodiametry, variance, and regular simplices from particle interactions.
 \newblock 
 {\em Archive Rational Mech. Analysis} {\bf 241} (2021) 553-576. https://doi.org/10.1007/s00205-021-01632-9
  

  \bibitem{McCann06}
Robert~J. McCann.
\newblock Stable rotating binary stars and fluid in a tube.
\newblock {\em Houston J. Math.}, {\bf 32} (2006) 603--632.

\bibitem{MogilnerEdelstein-Keshet99}
Alexander Mogilner and Leah Edelstein-Keshet.
\newblock A non-local model for a swarm.
\newblock {\em J. Math. Biol.}, {\bf 38} (1999) 534--570.

\bibitem{Patlak53}
Clifford S. Patlak.
\newblock Random walk with persistence and external bias. 
\newblock {\em Bull. Math. Biophys.} {\em 15} (1953), 311--338. 

\bibitem{Schwartz52}
Laurent Schwartz.
\newblock Transformation de Laplace des distributions.  
\newblock {\em Comm. S\'em. Math. Univ. Lund [Medd. Lunds Univ. Mat. Sem.]} (1952), Tome Supplmentaire, 196--206.
 
 \bibitem{Simione14}
 Robert Simione.
 \newblock Properties of Energy Minimizers of Nonlocal Interaction Energy.
 \newblock PhD Thesis, Carnegie Mellon University and Instituto Superior T\'ecnico, 2014.
 
 \bibitem{SteinShakarchi03}
Elias M.~Stein and Rami Shakarchi.
\newblock Complex Analysis. Princeton Lectures in Analysis {\bf 2}.
\newblock Princeton NJ: Princeton University Press, 2003.

\bibitem{SunUminskyBertozzi12}
 Hui Sun, David Uminsky, and Andrea L.~Bertozzi.
  \newblock {Stability and clustering of self-similar solutions of
              aggregation equations},
 \newblock  {\em J. Math. Phys.}, {\bf 53} (2012) 115610, 18.
 
\bibitem{TopazBertozziLewis06}
Chad~M. Topaz, Andrea~L. Bertozzi, and Mark~A. Lewis.
\newblock A nonlocal continuum model for biological aggregation.
\newblock {\em Bull. Math. Biol.}, {\bf 68} (2006) 1601--1623.

\bibitem{Toscani00}
Giuseppe Toscani.
\newblock One-dimensional kinetic models of granular flows.
\newblock {\em M2AN Math. Model. Numer. Anal.}, {\bf 34} (2000) 1277--1291.

\bibitem{JordanKinderlehrerOtto98}
Richard Jordan, David Kinderlehrer, Felix Otto.
\newblock The variational formulation of the Fokker-Planck equation. 
\newblock {\em SIAM J. Math. Anal.} 29 (1998), no. 1, 1--17.
 
  \bibitem{vonBrechtMcCalla14}
James H.~von Brecht and Scott G.~McCalla. 
\newblock Nonlinear stability through algebraically decaying point spectrum: applications to nonlocal interaction equations.
\newblock {\em SIAM J. Math. Anal.} {\bf 46} (2014), no. 6, 3727--3760.
  
\bibitem{vonBrechtUminskyKolokolnikovBertozzi12}
James~H. von Brecht, David Uminsky, Theodore Kolokolnikov, and Andrea~L.
  Bertozzi.
\newblock Predicting pattern formation in particle interactions.
\newblock {\em Math. Models Methods Appl. Sci.}, {\bf 22:} 1140002 (2012) 31.
 
  \bibitem{Villani03}
 C\'edric Villani.
 \newblock Topics in optimal transportation. 
 \newblock 
 American Mathematical Society, Providence, RI, 2003.

  \end{thebibliography}
\end{document}